\documentclass[11pt]{amsart}

\usepackage[a4paper,margin=30mm]{geometry}
\usepackage[T1]{fontenc}
\usepackage[utf8]{inputenc}
\usepackage{lmodern}
\usepackage{microtype}
\usepackage{amsmath,amssymb,amsthm,mathtools}
\usepackage{enumitem}
\usepackage{hyperref}

\hypersetup{
  colorlinks=true,
  linkcolor=blue,
  citecolor=blue,
  urlcolor=blue
}

\newtheorem{theorem}{Theorem}[section]
\newtheorem{proposition}[theorem]{Proposition}
\newtheorem{lemma}[theorem]{Lemma}
\newtheorem{corollary}[theorem]{Corollary}
\newtheorem{conjecture}[theorem]{Conjecture}
\newtheorem{problem}[theorem]{Problem}

\theoremstyle{definition}
\newtheorem{definition}[theorem]{Definition}
\theoremstyle{remark}
\newtheorem{remark}[theorem]{Remark}

\newcommand{\Z}{\mathbb Z}

\newcommand{\supp}{\operatorname{supp}}
\newcommand{\Hom}{\operatorname{Hom}}
\newcommand{\Prob}{\operatorname{Prob}}
\newcommand{\Ree}{\operatorname{Re}}

\title[A product-difference theorem in $\Z^2$]
{A Uniform Product-Difference Theorem\\
for Dense Subsets of $\mathbb Z^2$}
\author{Sayan Goswami}
\author{Chunlin Liu}
\thanks{This work was supported in part by the NBHM Postdoctoral Fellowship
(Reference No.~0204/27/(27)/2023/R \& D-II/11927), the Postdoctoral Fellowship
Program of the China Postdoctoral Science Foundation under Grant No.~BX20250067,
and the China Postdoctoral Science Foundation under Grant No.~2025M773074.}
\begin{document}

\begin{abstract}
We establish a uniform product-difference theorem for dense subsets of
$\mathbb Z^2$, which gives an affirmative answer to Problem~2 of Fish
 and, as consequences, to both parts of his Problem~1.
More precisely, we prove that for every $\delta>0$ there exists an integer
$K(\delta)\geq 1$ such that every set $E\subseteq\mathbb Z^2$ with upper
Banach density $d^\star(E)\geq\delta$ satisfies
\[
K(\delta)\mathbb Z
\subseteq
\{ab:(a,b)\in E-E\}.
\]
As consequences, we obtain affirmative answers to both parts of Fish's
Problem~1: for positive-density sets $E_1,E_2\subseteq\mathbb Z$ and
$E\subseteq\mathbb Z^2$, respectively, the sets
\[
(E_1-E_1)^2-(E_2-E_2)^2
\quad\text{and}\quad
\{x^2-y^2:(x,y)\in E-E\}
\]
contain nontrivial ideals of $\mathbb Z$, with generators depending only
on the corresponding density thresholds. In particular, the latter result also settles a conjecture of Davies concerning
differences of the indefinite quadratic form $x^2-y^2$ in dense
subsets of $\mathbb Z^2$.
\end{abstract}

\maketitle

\section{Introduction}

Estimating the quantitative and structural behavior of sumsets and
product sets lies at the heart of additive and arithmetic combinatorics.
For non-empty subsets $A,B$ of a commutative ring, we write
\[
   A+B=\{a+b:a\in A,\ b\in B\},\qquad
   A-B=\{a-b:a\in A,\ b\in B\},
\]
and
\[
   A\cdot B=\{ab:a\in A,\ b\in B\}.
\]
In expressions involving differences of squares, we use the pointwise
notation
\[
   A^2:=\{a^2:a\in A\},
\]
which should not be confused with the product set $A\cdot A$.

A fundamental theme in the subject is the tension between additive and
multiplicative structure: a finite set cannot be simultaneously too
structured with respect to both operations. This principle is embodied
in the celebrated Erd\H{o}s--Szemer\'edi sum--product conjecture
\cite{erdos}.

\begin{conjecture}[Erd\H{o}s--Szemer\'edi \cite{erdos}]
For every $\epsilon>0$ and every finite set $A\subseteq\mathbb Z$,
\[
   |A+A|+|A\cdot A|
   \gg_{\epsilon}|A|^{2-\epsilon}.
\]
\end{conjecture}

While the finite sum--product problem measures additive and
multiplicative growth through cardinality, in the infinite setting one
may instead ask what algebraic structure is forced by positive density.
In particular, a natural problem is to understand to what extent
additive largeness, as reflected by difference sets, necessarily
produces multiplicative structure. For general background on sumsets,
difference sets, and related structural questions, we refer to
Ruzsa \cite{ruzsa}.

Results of this type have been studied using methods from ergodic theory
and topological dynamics. Bj\"orklund and Fish \cite{fishold}, for
instance, established arithmetic patterns in difference sets by means
of ergodic-theoretic techniques. Of particular relevance to the present
paper is the following theorem of Fish
\cite[Corollary~1.1]{Fish2018}: if
$E_1,E_2\subseteq\mathbb Z$ have positive upper Banach density, then
there exists an integer $k\geq1$, bounded in terms of the two densities,
such that
\[
   k\mathbb Z\subseteq (E_1-E_1)(E_2-E_2).
\]
This product-difference phenomenon was subsequently extended by Liu
\cite{LiuPositiveDensity} to the setting of countable discrete amenable
groups.

Fish's theorem may be viewed geometrically as a statement about the
Cartesian product
\[
   E_1\times E_2\subseteq\mathbb Z^2.
\]
Indeed,
\[
   (E_1\times E_2)-(E_1\times E_2)
   =(E_1-E_1)\times(E_2-E_2),
\]
so the two coordinates of a difference vector arise independently from
two one-dimensional difference sets. For a general positive-density
set $E\subseteq\mathbb Z^2$, however, there is no such product structure:
the two coordinates of a vector in $E-E$ may be strongly correlated.
It is therefore a genuinely different problem to determine whether
positive density alone still forces the coordinate products of
differences to contain a nontrivial ideal of $\mathbb Z$.

To formulate this question precisely, we recall the standard notion of
upper Banach density in $\mathbb Z^2$.

\begin{definition}[Upper Banach density in $\mathbb Z^2$]
For $A\subseteq\mathbb Z^2$, define
\[
   d^\star(A)
   :=
   \lim_{N\to\infty}
   \sup_{z\in\mathbb Z^2}
   \frac{|A\cap(z+[-N,N]^2)|}{(2N+1)^2}.
\]
\end{definition}

For $E\subseteq\mathbb Z^2$, write
\[
   P(E):=\{xy:(x,y)\in E-E\}.
\]
Thus $P(E)$ records the coordinate products of difference vectors in
$E$. Motivated by the problem of removing the Cartesian-product
structure from the theorem above, Fish posed the following question.

\begin{problem}[{\cite[Problem~2]{Fish2018}}]\label{qn}
Let $E\subseteq\mathbb Z^2$ have positive upper Banach density.
Does there exist a constant $k_0$, depending only on $d^\star(E)$,
such that for some integer $1\leq k\leq k_0$ one has
\[
   k\mathbb Z\subseteq P(E)?
\]
\end{problem}

In the same paper, Fish posed two companion questions concerning
differences of squares.

\begin{problem}[{\cite[Problem~1]{Fish2018}}]\label{qn:squares}
\begin{enumerate}[label=\textup{(\roman*)}]
\item
Let $E_1,E_2\subseteq\mathbb Z$ have positive upper Banach density.
Is there a constant $k_0$, depending only on
$d^\star(E_1)$ and $d^\star(E_2)$, such that for some
$1\leq k\leq k_0$,
\[
   k\mathbb Z
   \subseteq
   (E_1-E_1)^2-(E_2-E_2)^2?
\]

\item
Let $E\subseteq\mathbb Z^2$ have positive upper Banach density.
Is there a constant $k_0$, depending only on $d^\star(E)$, such that
for some $1\leq k\leq k_0$,
\[
   k\mathbb Z
   \subseteq
   \{x^2-y^2:(x,y)\in E-E\}?
\]
\end{enumerate}
\end{problem}

A recent result of Goswami \cite{G} gives partial progress toward
Problem~\ref{qn} by ultrafilter methods. The main result of the present
paper gives a complete affirmative answer.

\begin{theorem}\label{thm:main}
For every $\delta>0$, there exists an integer $K(\delta)\geq1$ such
that, for every $E\subseteq\mathbb Z^2$ satisfying
$d^\star(E)\geq\delta$,
\[
   K(\delta)\mathbb Z
   \subseteq
   P(E)
   =
   \{xy:(x,y)\in E-E\}.
\]
\end{theorem}

In particular, taking $\delta=d^\star(E)$ gives an affirmative answer
to Problem~\ref{qn}. The formulation above is useful because the same
integer $K(\delta)$ works simultaneously for all sets whose upper
Banach density is at least $\delta$.

Theorem~\ref{thm:main} also answers both parts of
Problem~\ref{qn:squares}. The connection is provided by the elementary
factorization
\[
   x^2-y^2=(x-y)(x+y)
\]
and the linear map
\[
   T(x,y)=(x-y,x+y).
\]

\begin{corollary}\label{cor1}
Let $E_1,E_2\subseteq\mathbb Z$ satisfy
\[
   d^\star(E_1)=\alpha>0,
   \qquad
   d^\star(E_2)=\beta>0.
\]
Then there exists an integer $K=K(\alpha,\beta)\geq1$, depending only
on $\alpha$ and $\beta$, such that
\[
   K\mathbb Z
   \subseteq
   (E_1-E_1)^2-(E_2-E_2)^2.
\]
\end{corollary}

\begin{corollary}\label{cor2}
For every $\delta>0$, there exists an integer $K'(\delta)\geq1$ such
that every $E\subseteq\mathbb Z^2$ with
$d^\star(E)\geq\delta$ satisfies
\[
   K'(\delta)\mathbb Z
   \subseteq
   \{x^2-y^2:(x,y)\in E-E\}.
\]
\end{corollary}

\begin{remark}
    Corollary~\ref{cor2} also settles the density conjecture of Davies
\cite[Conjecture~6]{Davies2025}.
Indeed, write
\[
   \overline d(E)
   :=
   \limsup_{N\to\infty}
   \frac{|E\cap[-N,N]^2|}{(2N+1)^2}
\]
for the ordinary upper density of $E\subseteq\mathbb Z^2$.
Given $\epsilon>0$, put $d=K'(\epsilon)$. If
$\overline d(E)>\epsilon$, then
$d^\star(E)\geq\overline d(E)>\epsilon$, and
Corollary~\ref{cor2} gives
\[
   d^2\in d\mathbb Z
   \subseteq
   \{a^2-b^2:(a,b)\in E-E\}.
\]
Consequently, there exist $(x,y),(x',y')\in E$ such that
\[
   (x-x')^2-(y-y')^2=d^2,
\]
where $d$ depends only on $\epsilon$.
\end{remark}

For $N\geq1$, write $\mathbb Z_N=\mathbb Z/N\mathbb Z$.
The uniformity in Theorem~\ref{thm:main} yields finite cyclic
analogues by periodic lifting.

In the same-set case, Shkredov
\cite[Theorem~3.2]{Shkredov2024} obtained a related ideal containment
for squarefree moduli, with a bound that also depends on the number
of distinct prime divisors of the modulus. The following corollary
applies to arbitrary moduli and gives a bound depending only on the
density threshold.
\begin{corollary}\label{corfirst:abelian}
For every $\delta>0$, there exists an integer $K(\delta)\geq1$ such
that, for every $N\geq1$ and every $E\subseteq\mathbb Z_N^2$ satisfying
$
   |E|\geq\delta N^2,
$
there exists a divisor $d\mid N$ with $d\leq K(\delta)$ such that
\[
   d\mathbb Z_N
   \subseteq
   \{uv:(u,v)\in E-E\}.
\]
\end{corollary}

The same argument, together with Corollaries~\ref{cor1}
and~\ref{cor2}, gives the corresponding statements for differences
of squares.

\begin{corollary}\label{corsecond:abelian}
Let $\alpha,\beta>0$. There exists an integer
$K(\alpha,\beta)\geq1$ such that, for every $N\geq1$ and
$E_1,E_2\subseteq\mathbb Z_N$ satisfying
$
   |E_1|\geq\alpha N,
$ and  $
   |E_2|\geq\beta N,$
there exists a divisor $d\mid N$ with $d\leq K(\alpha,\beta)$ such that
\[
   d\mathbb Z_N
   \subseteq
   (E_1-E_1)^2-(E_2-E_2)^2.
\]
\end{corollary}

\begin{corollary}\label{corthird:abelian}
For every $\delta>0$, there exists an integer $K'(\delta)\geq1$ such
that, for every $N\geq1$ and every
$E\subseteq\mathbb Z_N^2$ satisfying
$
   |E|\geq\delta N^2,
$
there exists a divisor $d\mid N$ with $d\leq K'(\delta)$ such that
\[
   d\mathbb Z_N
   \subseteq
   \{u^2-v^2:(u,v)\in E-E\}.
\]
\end{corollary}

 The following rational analogues already follow from the
countable-field results of Bj\"orklund and Fish \cite{arxiv}.
We include their derivation from Theorem~\ref{thm:main} to explain
the finite-window transfer from the integer setting and the role
of rational scaling.

We use upper Banach density on the discrete amenable group
$(\mathbb Q^d,+)$. Thus, for $E\subseteq\mathbb Q^d$, we set
\[
   d^\star(E)
   :=
   \sup_{(F_n)}
   \limsup_{n\to\infty}
   \frac{|E\cap F_n|}{|F_n|},
\]
where the supremum is taken over all F\o{}lner sequences
$(F_n)$ in $\mathbb Q^d$.

\begin{corollary}\label{rational1}
If $E\subseteq\mathbb Q^2$ has positive upper Banach density, then
\[
   P(E)=\{xy:(x,y)\in E-E\}=\mathbb Q.
\]
\end{corollary}

\begin{corollary}\label{rational2}
If $E_1,E_2\subseteq\mathbb Q$ have positive upper Banach density, then
\[
   (E_1-E_1)^2-(E_2-E_2)^2=\mathbb Q.
\]
\end{corollary}

\begin{corollary}\label{rational3}
If $E\subseteq\mathbb Q^2$ has positive upper Banach density, then
\[
   \{x^2-y^2:(x,y)\in E-E\}=\mathbb Q.
\]
\end{corollary}

\medskip

\subsection*{Overview of the proof}

We briefly describe the mechanism behind
Theorem~\ref{thm:main}. To prove that a given non-zero integer $K$
belongs to $P(E)$, one has to find a difference vector
$(a,b)\in E-E$ such that
\[
   ab=K.
\]
Equivalently, one must show that the difference set $E-E$ intersects
the integer hyperbola
\[
   H_K:=\{(a,b)\in\mathbb Z^2:ab=K\}.
\]
The main difficulty is that upper Banach density is adapted to the
additive translation structure of $\mathbb Z^2$, whereas $H_K$ is a
sparse multiplicative configuration. Our proof constructs a sequence
of such hyperbolas, together with probability measures supported on
them whose Fourier transforms become asymptotically non-negative.
A spectral ratio bound then forces every sufficiently dense set to
have a difference on one of these hyperbolas. Finally, a stretching
argument upgrades the occurrence of one hyperbola to all integer
multiples of the same parameter.

The invariant-measure input is the classification of Bj\"orklund
and Fish \cite[Corollary~1.6]{arxiv}. The connection between
nonnegative Fourier transforms of invariant measures and
recurrence also appears in \cite[Lemma~8.5]{arxiv}.
Our argument obtains a lower bound uniform over frequencies,
realizes the resulting averages as finitely supported measures
on integer hyperbolas, and uses coordinate stretching to obtain
a common modulus. The proof proceeds in four steps.

\begin{enumerate}
\item[\textbf{Step 1.}]
\textbf{Reciprocal scaling on the rational solenoid.}
Let
\[
   X=\operatorname{Hom}(\mathbb Q,\mathbb T)
\]
be the Pontryagin dual of the discrete additive group $\mathbb Q$.
For $s\in\mathbb Q^\times$, consider the reciprocal scaling action
\[
   D_s(x,y)=(M_sx,M_{1/s}y)
\]
on $X^2$. The reciprocal form of this action reflects the algebraic
constraint that the product of the two coordinates remains fixed.

\item[\textbf{Step 2.}]
\textbf{Invariant-measure classification.}
We derive the classification from the disjointness theorem of
Einsiedler and Lindenstrauss for higher-rank actions by
automorphisms of solenoids, verifying the required algebraic
factor conditions. The classification implies that a suitable
test function has nonnegative integral against every probability
measure invariant under reciprocal scaling.

\item[\textbf{Step 3.}]
\textbf{F\o{}lner averaging and integer hyperbola measures.}
Averaging the reciprocal scaling action over explicit F\o{}lner sets
in $\mathbb Q^\times$ produces functions whose minima converge to
zero. After clearing denominators, these averages become Fourier
transforms of finitely supported symmetric probability measures
$\nu_j$ on $\mathbb Z^2$, supported on integer hyperbolas
\[
   ab=K_j,
\]
and satisfying
\[
   \widehat{\nu_j}\geq-\eta_j,
   \qquad
   \eta_j\longrightarrow0.
\]

\item[\textbf{Step 4.}]
\textbf{A Fourier ratio bound and stretching.}
A Hoffman-type Fourier ratio bound shows that a set
$E\subseteq\mathbb Z^2$ whose difference set avoids
$\operatorname{supp}\nu_j$ has upper Banach density at most
\[
   \frac{\eta_j}{1+\eta_j}.
\]
Choosing $j$ so that this quantity is smaller than the prescribed
density threshold forces $K_j\in P(E)$. Finally, the anisotropic
stretching
\[
   L_n(a,b)=(na,b)
\]
pushes the same measure to the hyperbola $ab=nK_j$ without changing
the Fourier lower bound. Hence every multiple of $K_j$ belongs to
$P(E)$.
\end{enumerate}

\section{Proof of Theorem \ref{thm:main}}

\subsection{Preliminaries on Upper Banach Density}

For $L\geq1$, write $Q_L=\{0,1,\ldots,L-1\}^2$. The density in the introduction can equivalently be computed as
\[
 d^\star(E)=\lim_{L\to\infty}\sup_{z\in\mathbb Z^2}
 \frac{|E\cap(z+Q_L)|}{L^2}.
\]
The next lemma relates this convention to the rectangular formulation used by Fish.

\begin{lemma}\label{lem:density-equivalence}
     For $E \subseteq \mathbb{Z}^2$, define the rectangular upper Banach density by$$d_{\text{rect}}^\star(E) = \limsup_{\min\{L,M\} \to \infty} \sup_{z \in \mathbb{Z}^2} \frac{\vert{}E \cap (z + (\{0,\dots,L-1\} \times \{0,\dots,M-1\}))\vert{}}{LM}.$$
Then $d_{\text{rect}}^\star(E) = d^\star(E)$. 
\end{lemma}

\begin{proof}
The inequality $d^\star(E)\leq d_{\mathrm{rect}}^\star(E)$ follows by taking $L=M$. For the converse, set $r=\lfloor\sqrt{\min\{L,M\}}\rfloor$ and tile an $L\times M$ rectangle by disjoint translates of $Q_r$, leaving a remainder $U$ with
\[
 \frac{|U|}{LM}\leq\frac rL+\frac rM=o(1).
\]
If $a_r=\sup_z|E\cap(z+Q_r)|/r^2$, then every translate of the rectangle contains at most $a_r LM+|U|$ points of $E$. Since $r\to\infty$, taking the supremum over translates and then the limit superior gives $d_{\mathrm{rect}}^\star(E)\leq d^\star(E)$.
\end{proof}

In view of Lemma~\ref{lem:density-equivalence}, we use the square
formulation throughout. We also record the density facts needed
for the changes of variables and the passage to $\mathbb Q^2$.
\begin{lemma}\label{lem:density-facts}
Let $G$ be a countable discrete amenable group and $A\subseteq G$.
For a nonempty finite set $F\subseteq G$, put
\[
   a_A(F)=\sup_{x\in G}\frac{|A\cap(x+F)|}{|F|}.
\]
Then
\[
   d_G^\star(A)
   =
   \inf_{\substack{F\subseteq G\\0<|F|<\infty}}a_A(F).
\]
In particular, every nonempty finite $F\subseteq G$ has a translate
$x+F$ satisfying
\[
   |A\cap(x+F)|\geq d_G^\star(A)|F|.
\]

Moreover, group isomorphisms preserve upper Banach density. If $G,H$
are countable discrete amenable groups and $A\subseteq G$, $B\subseteq H$,
then
\begin{equation}\label{eq:density-product}
       d_{G\times H}^\star(A\times B)
   =
   d_G^\star(A)d_H^\star(B).
\end{equation}
Finally, if $A\subseteq\mathbb Z^d$ and
$T\in M_d(\mathbb Z)$ has nonzero determinant, then
\[
   d_{\mathbb Z^d}^\star(TA)
   =
   \frac{d_{\mathbb Z^d}^\star(A)}{|\det T|}.
\]
\end{lemma}
\begin{proof}
The infimum formula for upper Banach density is standard; see e.g.,
\cite[Lemma~2.9]{DownarowiczHuczekZhang}. The finite-window assertion
follows immediately.

Invariance under group isomorphisms is immediate from the infimum
formula. For product sets, if $F\subseteq G$ and $H_0\subseteq H$ are
finite and nonempty, then
\[
   a_{A\times B}(F\times H_0)
   =
   a_A(F)a_B(H_0).
\]
Taking infima gives
\[
   d_{G\times H}^\star(A\times B)
   \leq
   d_G^\star(A)d_H^\star(B).
\]
Conversely, for F\o{}lner sequences $(F_N)$ in $G$ and $(H_N)$ in $H$,
the sequence $(F_N\times H_N)$ is F\o{}lner in $G\times H$, and hence
\[
   d_{G\times H}^\star(A\times B)
   =
   \lim_{N\to\infty}
   a_{A\times B}(F_N\times H_N)
   =
   d_G^\star(A)d_H^\star(B).
\]

For the final assertion, let
\[
   \Lambda=T\mathbb Z^d,
   \qquad
   [\mathbb Z^d:\Lambda]=|\det T|.
\]
The map $T:\mathbb Z^d\to\Lambda$ is a group isomorphism, so
\[
   d_\Lambda^\star(TA)=d_{\mathbb Z^d}^\star(A).
\]
Since $\Lambda$ has index $|\det T|$ in $\mathbb Z^d$, a subset
$D\subseteq\Lambda$ satisfies
\[
   d_{\mathbb Z^d}^\star(D)
   =
   \frac1{|\det T|}d_\Lambda^\star(D).
\]
Applying this to $D=TA$ gives the result.
\end{proof}

\subsection{The Rational Solenoid and Pontryagin Duality}

Let $\mathbb T=\mathbb R/\mathbb Z$ and let
\[
 X=\Hom(\mathbb Q,\mathbb T)
\]
be the Pontryagin dual of the discrete additive group $\mathbb Q$, with the topology of pointwise convergence. For $q\in\mathbb Q$ and $s\in\mathbb Q^\times$, define
\[
 e_q(x)=\exp(2\pi i x(q)),\qquad (M_sx)(q)=x(sq).
\]
We record the basic properties used below.

\begin{lemma}\label{lem:rational-solenoid}
    The group $X$ is an infinite, compact, connected, metrizable abelian group. Its Pontryagin dual is canonically isomorphic to $\mathbb{Q}$, and its continuous characters are precisely the functions $e_q$ for $q \in \mathbb{Q}$. Moreover, the following properties hold:
    \begin{enumerate}
        \item Each scaling map $M_s$ is a continuous group automorphism, and the composition satisfies $M_s M_t = M_{st}$.

        \item On the dual space $\widehat{X} = \mathbb{Q}$, the induced dual map of $M_s$ is simply multiplication by $s$.

        \item The normalized Haar measure $m_X$ on $X$ is non-atomic and  $M_s$-invariant.

        \item The characters form an orthonormal system with respect to Haar measure:

                        $$\int_X e_q \, dm_X = \begin{cases} 1, & q = 0, \\ 0, & q \neq 0. \end{cases}$$
        \item Any finite Borel measure on $X$ is uniquely determined by its integrals against the family of characters $\{e_q\}_{q \in \mathbb{Q}}$.                 
    \end{enumerate}

\end{lemma}

\begin{proof}
The group $X$ is a closed subgroup of the countable product $\mathbb T^{\mathbb Q}$, hence compact and metrizable. The coordinates $t_n=x(1/n!)$ identify it with
\[
 \varprojlim\bigl(\mathbb T\xleftarrow{\times2}\mathbb T
 \xleftarrow{\times3}\mathbb T\xleftarrow{\times4}\cdots\bigr).
\]
Indeed, $t_n=(n+1)t_{n+1}$, and these relations define $x$ on every rational number. The bonding maps are surjective, so this inverse limit is connected. The maps $x_t(q)=tq\pmod1$, $t\in\mathbb R$, show that each nonzero coordinate projection is onto $\mathbb T$; in particular, $X$ is infinite. Pontryagin duality identifies $\widehat X$ with $\mathbb Q$, with $q$ corresponding to $e_q$.

The identities $M_sM_t=M_{st}$ and $M_s^{-1}=M_{1/s}$, together with coordinatewise continuity, prove (1). The relation $e_q\circ M_s=e_{sq}$ proves (2). Every continuous automorphism preserves normalized Haar measure. Moreover, Haar measure on an infinite compact group is non-atomic, since all singleton masses are equal by translation invariance. This proves (3).

For $q\neq0$, choose $y\in X$ with $e_q(y)\neq1$. Translation invariance gives
\[
 \int_X e_q\,dm_X=e_q(y)\int_X e_q\,dm_X,
\]
so the integral vanishes. The integral of $e_0=1$ is one, proving (4). Finally, the linear span of the characters is a self-adjoint unital algebra separating points of $X$. Stone--Weierstrass and uniqueness of finite Borel measures from their continuous integrals prove (5).
\end{proof}

The group $X$ is the one-dimensional rational solenoid. In the terminology used in \cite{EinsiedlerLindenstrauss2022}, a \emph{solenoid} is a compact connected abelian group whose Pontryagin dual embeds in a finite-dimensional $\mathbb Q$-vector space; the definition is not restricted to rank one.

\subsection{Invariant Marginals and Reciprocal Disjointness}\label{section disjoint}

We first classify the probability measures on $X$ invariant under every $M_s$, $s\in\mathbb Q^\times$.

\begin{lemma}\label{lem:marginals}
    Let $\rho \in \text{Prob}(X)$ be a probability measure satisfying$$(M_s)_* \rho = \rho \quad \text{for all } s \in \mathbb{Q}^\times.$$
Then there exists a unique constant $c \in [0, 1]$ such that
  \begin{equation}\label{eq:marginal-classification}
 \rho=c\delta_0+(1-c)m_X.
\end{equation}
Moreover, this constant represents the atomic mass at the identity: $c = \rho(\{0\})$.  
\end{lemma}

\begin{proof}
Write $\widehat\rho(q)=\int_X e_q\,d\rho$. Since $e_q\circ M_s=e_{sq}$, invariance implies that $\widehat\rho(q)=c$ is constant for $q\neq0$. The identity $\widehat\rho(-q)=\overline{\widehat\rho(q)}$ shows that $c$ is real, and $|c|\leq1$.

For any distinct $q_1,\ldots,q_N\in\mathbb Q$,
\[
 0\leq\int_X\left|\sum_{i=1}^N e_{q_i}\right|^2d\rho
   =N+N(N-1)c.
\]
Letting $N\to\infty$ gives $c\geq0$. Thus $c\delta_0+(1-c)m_X$ is a probability measure with the same Fourier coefficients as $\rho$. Lemma~\ref{lem:rational-solenoid}(5) gives \eqref{eq:marginal-classification}. Since $m_X$ is non-atomic, $c=\rho(\{0\})$, which also proves uniqueness.
\end{proof}

 We recall the relevant terminology:
    \begin{enumerate}
        \item An action $\gamma$ of $\mathbb{Z}^d$ is {\it virtually cyclic} if there exists a finite-index subgroup $\Lambda \le \mathbb{Z}^d$ and a vector $v \in \mathbb{Z}^d$ such that for every $n \in \Lambda$, one has $\gamma^n = \gamma^{kv}$ for some $k \in \mathbb{Z}$. 

        \item An algebraic factor map is a surjective continuous group homomorphism that intertwines two actions with the same parameter (with no reparameterization of $\mathbb{Z}^d$ permitted).  
    \end{enumerate}
A \emph{joining} of measure-preserving actions with the same acting group is an invariant probability measure on their product whose coordinate marginals are the given invariant measures. The actions are \emph{mutually disjoint} if their product measure is the only joining.

\begin{theorem}[{Einsiedler--Lindenstrauss \cite[Corollary~1.4, general case]{EinsiedlerLindenstrauss2022}}]\label{thm:EL-disjointness}
Let $d,r\geq2$, and let $\alpha_i$, $1\leq i\leq r$, be $\mathbb Z^d$-actions by automorphisms of solenoids, each having no nontrivial virtually cyclic algebraic factor. The actions are not mutually disjoint with respect to Haar measure if and only if there exist distinct $i,j$, a finite-index subgroup $\Lambda\leq\mathbb Z^d$, and a nontrivial $\Lambda$-action on a solenoid that is an algebraic factor of both $\alpha_i|_\Lambda$ and $\alpha_j|_\Lambda$.
\end{theorem}

We deploy this theorem on $X^2$ using two specific actions generated by mutually reciprocal multipliers. We define two $\mathbb{Z}^2$-actions on $X$ by:$$\alpha^{(a,b)} = M_{2^a 3^b}, \quad \text{and} \quad \beta^{(a,b)} = M_{2^{-a} 3^{-b}}.$$

\begin{proposition}\label{eq:alpha-beta-actions}
    The Haar systems $(X, m_X, \alpha)$ and $(X, m_X, \beta)$ are disjoint. 
\end{proposition}

\begin{proof}
  We check the two requirements of Theorem~\ref{thm:EL-disjointness}.

First, let $\pi:X\to Y$ be a nontrivial algebraic factor of $\alpha$. Its dual is an injection $\iota:\widehat Y\hookrightarrow\mathbb Q$. If $(a,b)\in\mathbb Z^2$ acts trivially on $Y$, then
\[
 (2^a3^b-1)\iota(h)=0\qquad(h\in\widehat Y).
\]
Taking $h\neq0$ gives $2^a3^b=1$, hence $a=b=0$. Thus every nontrivial factor action is faithful. A faithful $\mathbb Z^2$-action cannot be virtually cyclic: its restriction to a finite-index subgroup is still injective, whereas that subgroup cannot embed in a cyclic group. The same argument applies to $\beta$.

Second, let $\Lambda\leq\mathbb Z^2$ have finite index, and choose $l\geq1$ with $(l,0)\in\Lambda$. Suppose a nontrivial $\Lambda$-action $\gamma$ on a solenoid $Y$ is a common algebraic factor of $\alpha|_\Lambda$ and $\beta|_\Lambda$. The two factor maps induce injections
\[
 \iota_+,\iota_-:\widehat Y\longrightarrow\mathbb Q.
\]
If $T$ is the dual automorphism of $\gamma^{(l,0)}$, equivariance gives
\[
 \iota_+(Th)=2^l\iota_+(h),\qquad
 \iota_-(Th)=2^{-l}\iota_-(h).
\]
Because $2^l$ is an integer, additivity and injectivity of $\iota_+$ imply $Th=2^lh$. Substitution into the second identity yields
\[
 (2^{2l}-1)\iota_-(h)=0\qquad(h\in\widehat Y).
\]
Hence $\widehat Y=\{0\}$, a contradiction. Theorem~\ref{thm:EL-disjointness} now gives disjointness. 
\end{proof}

\subsection{Classification under Reciprocal Rational Scaling}
For $s\in\mathbb Q^\times$, define the reciprocal scaling automorphism $D_s:X^2\to X^2$ by
\begin{equation}\label{eq:Ds-definition}
 D_s(x,y)=(M_sx,M_{1/s}y),\quad \text{for } (x,y)\in X^2.
\end{equation}

We have $D_sD_t=D_{st}$. The test function used below is  \begin{equation}\label{eq:C-definition}
 C(x,y)=\Ree\bigl(e_1(x)e_1(y)\bigr)
       =\cos\bigl(2\pi(x(1)+y(1))\bigr).
\end{equation}

    The following classification is a specialization of
\cite[Corollary~1.6]{arxiv} to $\mathbb Q$. We nevertheless include a
proof based on the preceding disjointness result, since the argument
makes explicit the four invariant measures that enter the Fourier
analysis below.

\begin{proposition}\label{prop:joint-classification}
   Let $\sigma \in \text{Prob}(X^2)$ be a probability measure satisfying$$(D_s)_* \sigma = \sigma \quad \text{for all } s \in \mathbb{Q}^\times.$$
Then $\sigma$ is a convex combination of the following four canonical product measures:
  $$\delta_0 \otimes \delta_0, \quad m_X \otimes \delta_0, \quad \delta_0 \otimes m_X, \quad \text{and} \quad m_X \otimes m_X.$$
In particular,
\[
   \int_{X^2} C\,d\sigma
   =
   \sigma(\{(0,0)\})
   \geq0.
\]
\end{proposition}

\begin{proof}
Put $X^\times=X\setminus\{0\}$ and partition $X^2$ into the four invariant Borel sets
\[
 \begin{aligned}
 \Omega_{00}&=\{0\}\times\{0\},&
 \Omega_{10}&=X^\times\times\{0\},\\
 \Omega_{01}&=\{0\}\times X^\times,&
 \Omega_{11}&=X^\times\times X^\times.
 \end{aligned}
\]
Whenever $\sigma(\Omega_{ij})>0$, let $\sigma_{ij}$ be the normalized restriction to $\Omega_{ij}$. Each coordinate marginal of $\sigma_{ij}$ is invariant under all $M_s$: on the second coordinate this follows because $s\mapsto s^{-1}$ permutes $\mathbb Q^\times$.

Lemma~\ref{lem:marginals} shows that a marginal supported at $0$ is $\delta_0$, while a marginal assigning no mass to $0$ is $m_X$. Consequently,
\[
 \sigma_{00}=\delta_0\otimes\delta_0,\quad
 \sigma_{10}=m_X\otimes\delta_0,\quad
 \sigma_{01}=\delta_0\otimes m_X
\]
whenever the corresponding restriction is defined. The measure $\sigma_{11}$ has both marginals equal to $m_X$ and is invariant, in particular, under $D_{2^a3^b}=\alpha^{(a,b)}\times\beta^{(a,b)}$. Proposition~\ref{eq:alpha-beta-actions} therefore gives $\sigma_{11}=m_X\otimes m_X$.

Recombining the restrictions proves the asserted convex combination. By character orthogonality, the integral of $C$ is zero for the last three product measures and one for $\delta_0\otimes\delta_0$. Since $m_X(\{0\})=0$, the coefficient of this last measure is $\sigma(\{(0,0)\})$, as required.
\end{proof}

\subsection{F\o{}lner Averaging in $\mathbb{Q}^\times$}\label{F}

Let $p_1<p_2<\cdots$ be the primes. For $j\geq1$, define
\[
 F_j^+=\left\{\prod_{i=1}^j p_i^{a_i}:a_i\in\mathbb Z,\ -j\leq a_i\leq j\right\},
 \qquad F_j=F_j^+\cup(-F_j^+).
\]
Unique prime factorization gives $|F_j^+|=(2j+1)^j$ and $|F_j|=2(2j+1)^j$. We use these sets to average over the multiplicative group $\mathbb Q^\times$.

  \begin{lemma}\label{lem:Fj-Folner}
       The sequence $(F_j)_{j \ge 1}$ constitutes a F\o{}lner sequence in the multiplicative abelian group $\mathbb{Q}^\times$.
  \end{lemma}

\begin{proof}
Fix $s=\epsilon\prod_{i=1}^L p_i^{b_i}\in\mathbb Q^\times$, with $\epsilon\in\{1,-1\}$ and $b_i\in\mathbb Z$. For $j\geq\max\{L,|b_1|,\ldots,|b_L|\}$, the sign only permutes the two halves of $F_j$, while multiplication shifts the first $L$ prime exponents by $b_i$. Hence
\[
 \frac{|F_j\cap sF_j|}{|F_j|}
 =\prod_{i=1}^L\left(1-\frac{|b_i|}{2j+1}\right).
\]
Using $1-\prod_i(1-u_i)\leq\sum_i u_i$ for $0\leq u_i\leq1$, we obtain
\[
 \frac{|sF_j\triangle F_j|}{|F_j|}
 =2\left(1-\frac{|F_j\cap sF_j|}{|F_j|}\right)
 \leq\frac{2\sum_{i=1}^L|b_i|}{2j+1}\longrightarrow0.\qedhere
\]
\end{proof}

Define
\[
 \Phi_j(x,y)=\frac1{|F_j|}\sum_{r\in F_j}C(D_r(x,y)),
 \qquad \eta_j=-\min_{(x,y)\in X^2}\Phi_j(x,y).
\]
The minimum exists because $\Phi_j$ is continuous on the compact space $X^2$. Thus $\Phi_j\geq-\eta_j$ everywhere.

\begin{proposition}\label{prop:eta-to-zero}
    For every integer $j\ge 1$, one has $\eta_j \ge 0$, and the sequence of minimal values satisfies$$\lim_{j \to \infty} \eta_j = 0.$$
\end{proposition}

\begin{proof}
Since each $D_r$ preserves $m_X\otimes m_X$ and $\int C\,d(m_X\otimes m_X)=0$, we have $\int\Phi_j\,d(m_X\otimes m_X)=0$. Consequently, $\min\Phi_j\leq0$ and $\eta_j\geq0$.

Suppose that $\eta_j\not\to0$. Choose $\epsilon>0$, indices $j_t\to\infty$, and points $z_t\in X^2$ with $\Phi_{j_t}(z_t)\leq-\epsilon$. Define
\[
 \sigma_t=\frac1{|F_{j_t}|}\sum_{r\in F_{j_t}}\delta_{D_rz_t}.
\]
After taking a subsequence, weak-* compactness of $\Prob(X^2)$ gives $\sigma_t\to\sigma$. For every fixed $s\in\mathbb Q^\times$ and $f\in C(X^2)$,
\[
 \left|\int f\circ D_s\,d\sigma_t-\int f\,d\sigma_t\right|
 \leq\|f\|_\infty\frac{|sF_{j_t}\triangle F_{j_t}|}{|F_{j_t}|}
 \longrightarrow0.
\]
The estimate is independent of $z_t$, which may vary with $t$. Passing to the limit shows that $\sigma$ is invariant under every $D_s$. Proposition~\ref{prop:joint-classification} then gives $\int C\,d\sigma\geq0$, whereas weak-* convergence gives
\[
 \int C\,d\sigma=\lim_{t\to\infty}\Phi_{j_t}(z_t)\leq-\epsilon.
\]
This contradiction proves $\eta_j\to0$.
\end{proof}

\subsection{Integer Hyperbola Measures}
We now clear the denominators in the finite averages to obtain measures supported on integer hyperbolas.

For any integer $j \ge 1$, we recall that $p_1 < p_2 < \dots < p_j$ are the first $j$ prime numbers. We define their product by$$P_j = \prod_{i=1}^j p_i.$$

Using this product, we define the base scaling factor $A_j$ and the hyperbola constant $K_j$ by setting
  $$A_j = P_j^j, \quad \text{and} \quad K_j = A_j^2 = P_j^{2j}.$$
We also define the combinatorial count $N_j = (2j+1)^j$.

Every positive divisor of $K_j=\prod_{i=1}^j p_i^{2j}$ is uniquely of the form $\prod_{i=1}^j p_i^{c_i}$ with $0\leq c_i\leq2j$. Thus $K_j$ has exactly $N_j=(2j+1)^j$ positive divisors.

Place equal mass on the two branches of the integer hyperbola by defining
\begin{equation}\label{discrete}
    \nu_j = \frac{1}{2N_j} \sum_{d\vert{}K_j, d>0} \left( \delta_{(d, K_j/d)} + \delta_{(-d, -K_j/d)} \right).
\end{equation}

Every integer solution of $ab=K_j>0$ is either $(d,K_j/d)$ or $(-d,-K_j/d)$ for a positive divisor $d$ of $K_j$. Therefore,
\[
 \supp\nu_j=H_{K_j}:=\{(a,b)\in\mathbb Z^2:ab=K_j\}.
\]

  For any finitely supported measure $\nu$ on $\mathbb{Z}^2$, we adopt the standard convention for the discrete Fourier transform:$$\widehat{\nu}(\theta_1, \theta_2) = \sum_{(a,b) \in \mathbb{Z}^2} \nu(a,b) \exp(2\pi i (a\theta_1 + b\theta_2)), \quad \text{for } (\theta_1, \theta_2) \in \mathbb{T}^2.$$

The next proposition identifies the Fourier polynomial with the solenoidal average, including equality of their minima.
  
  \begin{proposition}\label{prop:hyperbola-transform}
    The transform $\widehat\nu_j$ is real-valued. For all $x,y\in X$,
\begin{equation}\label{eq:average-fourier-identity}
 \Phi_j(x,y)=\widehat\nu_j\bigl(x(1/A_j),y(1/A_j)\bigr).
\end{equation}
Moreover,
\begin{equation}\label{eq:exact-spectral-minimum}
 \eta_j=-\min_{\theta\in\mathbb T^2}\widehat\nu_j(\theta).
\end{equation}
In particular, $\widehat\nu_j\geq-\eta_j$ on $\mathbb T^2$, with $\eta_j\geq0$ and $\eta_j\to0$.
  \end{proposition}

     \begin{proof}
The map $r\mapsto d=A_jr$ is a bijection from $F_j^+$ to the positive divisors of $K_j$: it changes each prime exponent from $a_i\in[-j,j]$ to $j+a_i\in[0,2j]$. Under this bijection,
\[
 r=\frac d{A_j},\qquad \frac1r=\frac{K_j/d}{A_j}.
\]
For $x,y\in X$, set $\theta_1=x(1/A_j)$ and $\theta_2=y(1/A_j)$ in $\mathbb T$. Additivity gives $x(r)=d\theta_1$ and $y(1/r)=(K_j/d)\theta_2$. The contributions of $r$ and $-r$ to the cosine average agree. Hence
\[
 \Phi_j(x,y)=\frac1{N_j}\sum_{\substack{d\mid K_j\\d>0}}
 \cos\!\left(2\pi\left(d\theta_1+\frac{K_j}{d}\theta_2\right)\right)
 =\widehat\nu_j(\theta_1,\theta_2),
\]
where the last equality follows from the symmetry of \eqref{discrete}.

The evaluation map $X\to\mathbb T$, $x\mapsto x(1/A_j)$, is onto: given a real representative $t$ of $\theta\in\mathbb T$, the homomorphism $x(q)=A_jtq\pmod1$ has $x(1/A_j)=\theta$. Therefore the pair of evaluations ranges over all of $\mathbb T^2$, proving \eqref{eq:exact-spectral-minimum}. Proposition~\ref{prop:eta-to-zero} supplies the remaining conclusions.
\end{proof}

 \subsection{A Fourier Ratio Bound for Upper Banach Density}

We use the following Hoffman-type spectral ratio bound. It is a finite-support version of the operator bounds for independence ratios; see, for example, \cite{BachocEtAl2014} for a general operator framework. We give the finite-quotient and periodization argument needed here.

 \begin{proposition}\label{prop:ratio bound}
      Let $\nu$ be a finitely supported, symmetric probability measure on $\mathbb{Z}^2$. Suppose that the origin does not lie in its support, meaning $0 \notin S := \text{supp } \nu$. Assume that its Fourier transform satisfies a uniform lower bound
  $$\widehat{\nu}(\theta) \ge -\eta, \quad \text{for all } \theta \in \mathbb{T}^2,$$
for some constant $\eta \ge 0$. If a subset $E \subseteq \mathbb{Z}^2$ satisfies
  $$(E-E) \cap S = \emptyset,$$
then its upper Banach density is bounded strictly by
  $$d^\star(E) \le \frac{\eta}{1+\eta}.$$  
 \end{proposition}
\begin{proof}
If $d^\star(E)=0$, the conclusion is immediate. We therefore assume $d^\star(E)>0$.

    We proceed in two stages: first, we establish the bound for sets that are periodic, and then we utilize a geometric ``padding'' argument to transfer this bound to arbitrary sets with positive upper Banach density.

\noindent\textbf{Step $1$: The Periodic Case.}
Suppose initially that the set $E$ is periodic modulo some lattice $q\mathbb{Z}^2$. We can then project our analysis down to the finite quotient group $G_q = (\mathbb{Z}/q\mathbb{Z})^2$. Let $f : G_q \to \{0,1\}$ denote the indicator function of $E$ on this finite torus. We equip the function space $L^2(G_q)$ with the normalized counting measure, so the global density of $E$ is given by the expectation
  $$\rho = \frac{1}{q^2} \sum_{x \in G_q} f(x).$$
We define the discrete convolution operator $T_\nu : L^2(G_q) \to L^2(G_q)$ by taking the weighted average of $f$ translated by the elements of $S$:
  $$(T_\nu f)(x) = \sum_{s \in S} \nu(s) f(x+s).$$

Because $E$ contains  no pairs of points separated by a vector in $S$, the product $f(x)f(x+s)$ must be exactly $0$ for every $s \in S$ and every $x \in G_q$. Consequently, the standard $L^2$ inner product vanishes entirely:
  $$\langle f, T_\nu f \rangle = \frac{1}{q^2} \sum_{x \in G_q} f(x) (T_\nu f)(x) = 0.$$

  To exploit this orthogonality, we apply the spectral decomposition of $T_\nu$. The standard exponential characters of $G_q$ form an orthonormal eigenbasis for $T_\nu$. By direct computation, the eigenvalue associated with a frequency $\xi \in G_q$ is exactly the discrete Fourier transform evaluated at the rational frequency, $\widehat{\nu}(\xi/q)$. Because $\nu$ is a symmetric measure, $T_\nu$ is a self-adjoint operator, guaranteeing purely real eigenvalues. The trivial character (the constant function $\mathbf{1}$) corresponds to the eigenvalue $\widehat{\nu}(0) = 1$, since $\nu$ is a probability measure. Furthermore, by our primary hypothesis, all eigenvalues of $T_\nu$ are bounded below by $-\eta$.  We now decompose our indicator function $f$ into its constant expectation and an orthogonal mean-zero fluctuation: $f = \rho \mathbf{1} + g$, where $\int g = 0$. Expanding the vanishing inner product yields:
  $$0 = \langle f, T_\nu f \rangle = \langle \rho \mathbf{1} + g, T_\nu (\rho \mathbf{1} + g) \rangle = \rho^2 \langle \mathbf{1}, T_\nu \mathbf{1} \rangle + \langle g, T_\nu g \rangle = \rho^2 + \langle g, T_\nu g \rangle.$$

Since $g$ is a linear combination of non-trivial characters, and every non-trivial eigenvalue is at least $-\eta$, we can bound the quadratic form $\langle g, T_\nu g \rangle$ from below by $-\eta \Vert{}g\Vert{}_2^2$. This gives:
\begin{equation}\label{proof:ineq}
    0 \ge \rho^2 - \eta \Vert{}g\Vert{}_2^2.
\end{equation}

Because $f$ is an indicator function taking values in $\{0,1\}$, its $L^2$ norm satisfies $\Vert{}f\Vert{}_2^2 = \rho$. By Pythagoras' theorem on orthogonal components,
\begin{equation}\label{proof:ortho}
    \Vert{}g\Vert{}_2^2 = \Vert{}f\Vert{}_2^2 - \Vert{}\rho \mathbf{1}\Vert{}_2^2 = \rho - \rho^2 = \rho(1-\rho).
\end{equation}
From inequalities \eqref{proof:ineq} and \eqref{proof:ortho} we have
  $$0 \ge \rho^2 - \eta\rho(1-\rho).$$
Since $\rho=d^\star(E)>0$ in the periodic case under consideration, division by $\rho$ gives $\rho\leq\eta/(1+\eta)$.

\noindent\textbf{Step $2$: The General Case.} Now, let $E$ be an arbitrary set in $\mathbb{Z}^2$. We must extract a dense periodic approximation. Define the maximum radius of the support of $\nu$ in the uniform norm:$$R = \max_{s \in S} \Vert{}s\Vert{}_\infty.$$
Fix an arbitrary real density threshold $0 < \gamma < d^\star(E)$. By the definition of the upper Banach density, there exist arbitrarily large odd integers $L$ and corresponding base translations $z \in \mathbb{Z}^2$ such that the intersection of $E$ with the square $z + \{0, \dots, L-1\}^2$ has cardinality strictly greater than $\gamma L^2$.  Translating this dense local intersection back to the standard square origin gives a finite block of elements:$$B \subseteq \{0, \dots, L-1\}^2 \quad \text{with} \quad \vert{}B\vert{} > \gamma L^2.$$

Because $B$ is merely a shifted subset of $E$, it naturally inherits the property that $(B-B) \cap S = \emptyset$. We now periodize this finite block. To prevent differences between distinct period cells from falling into $S$, we introduce a sufficiently wide geometric padding. Set the new period length to
  $$q = L + 2R + 1,$$
and define the fully periodic set
  $$E' = B + q\mathbb{Z}^2.$$
We must verify that $(E'-E') \cap S = \emptyset$. Consider a difference $u - v$ where $u, v \in E'$. If $u$ and $v$ belong to the exact same period cell, their difference belongs to $B-B$, which we already know is disjoint from $S$. If $u$ and $v$ belong to distinct period cells, then the absolute value of at least one of their coordinates must span the gap between the blocks. Since each block $B$ is confined to a square of length $L-1$, the distance between points in adjacent cells is bounded below by the period minus the maximum internal block width:
  $$q - (L-1) = (L + 2R + 1) - L + 1 = 2R + 2 > R.$$
However, every element of $S$ has a supremum norm at most $R$. Thus, cross-cell differences are strictly too large to belong to $S$, guaranteeing that $(E'-E') \cap S = \emptyset$.  Since $E'$ is periodic modulo $q\mathbb{Z}^2$, we can apply the result of Stage 1 to conclude that its density over the period cell is bounded by the spectral ratio:$$\frac{\vert{}B\vert{}}{q^2} \le \frac{\eta}{1+\eta}.$$
Using our lower bound $\vert{}B\vert{} > \gamma L^2$, we obtain:
  $$\gamma \left( \frac{L}{q} \right)^2 < \frac{\eta}{1+\eta}.$$

Recall that this holds for arbitrarily large values of $L$. Since $q = L + 2R + 1$ and $R$ is a fixed constant depending only on the measure $\nu$, the geometric ratio $L/q \to 1$ as $L \to \infty$. Taking this limit yields $\gamma \le \frac{\eta}{1+\eta}$. Finally, because this inequality holds for any $\gamma$ strictly less than $d^\star(E)$, we may take the limit supremum as $\gamma \uparrow d^\star(E)$ to conclude $d^\star(E) \le \frac{\eta}{1+\eta}$, completing the proof. 
\end{proof}

\subsection{Stretching the Measure and Proving the Main Theorem}

We have a bound of $\eta_j/(1+\eta_j)$ on the density of sets whose differences avoid $ab=K_j$. To obtain the whole ideal, we push the measure forward to $ab=nK_j$, keeping the same Fourier lower bound for every $n\neq0$.

For $n\in\mathbb Z\setminus\{0\}$, define
\[
 L_n:\mathbb Z^2\to\mathbb Z^2,\qquad L_n(a,b)=(na,b).
\]

\begin{lemma}\label{str}
    Let $\nu$ be a finitely supported, symmetric probability measure on $\mathbb{Z}^2$ whose discrete Fourier transform satisfies the uniform lower bound$$\widehat{\nu}(\theta_1, \theta_2) \ge -\eta, \quad \text{for all } (\theta_1, \theta_2) \in \mathbb{T}^2.$$
Define the pushforward measure $\nu^{(n)} = (L_n)_* \nu$. Then $\nu^{(n)}$ is also a finitely supported, symmetric probability measure, and its Fourier transform satisfies the exact same spectral bound:
  $$\widehat{\nu^{(n)}}(\theta_1, \theta_2) = \widehat{\nu}(n\theta_1, \theta_2) \ge -\eta.$$
Furthermore, if every vector in the support of $\nu$ has a coordinate product equal to some non-zero constant $K \neq 0$, then every vector in the support of $\nu^{(n)}$ has a coordinate product equal to $nK$. 
\end{lemma}

\begin{proof}
The pushforward preserves total mass and finite support. Since $L_n(-v)=-L_n(v)$, it also preserves symmetry. Directly from the definition,
\begin{align*}
 \widehat{\nu^{(n)}}(\theta_1,\theta_2)
 &=\sum_{(a,b)\in\mathbb Z^2}\nu(a,b)
       \exp\bigl(2\pi i(na\theta_1+b\theta_2)\bigr)\\
 &=\widehat\nu(n\theta_1,\theta_2)\geq-\eta.
\end{align*}
If $ab=K$ on $\supp\nu$, then $(na)b=nK$ on the pushforward support.
\end{proof}

\begin{proof}[Proof of Theorem~\ref{thm:main}]
For $\delta>1$ there are no sets satisfying the density assumption, so take $K(\delta)=1$. Assume $0<\delta\leq1$. By Proposition~\ref{prop:eta-to-zero}, choose $j=j(\delta)$ such that
\begin{equation}\label{proofneed}
 \frac{\eta_j}{1+\eta_j}<\delta,
\end{equation}
and put
\[
 K(\delta)=K_j=\left(\prod_{i=1}^j p_i\right)^{2j}.
\]
Let $E\subseteq\mathbb Z^2$ satisfy $d^\star(E)\geq\delta$, and fix any nonzero integer $n$. By Proposition~\ref{prop:hyperbola-transform} and Lemma~\ref{str}, the probability measure $\nu_j^{(n)}$ is symmetric, finitely supported, has Fourier transform at least $-\eta_j$, and is supported on $ab=nK(\delta)$. In particular, its support does not contain the origin.

If $nK(\delta)\notin P(E)$, then $(E-E)\cap\supp\nu_j^{(n)}=\varnothing$. Proposition~\ref{prop:ratio bound} and \eqref{proofneed} give
\[
 d^\star(E)\leq\frac{\eta_j}{1+\eta_j}<\delta,
\]
a contradiction. Thus every nonzero multiple of $K(\delta)$ belongs to $P(E)$. Finally, $E\neq\varnothing$, so $(0,0)\in E-E$ and $0\in P(E)$. The chosen $j$, and hence $K(\delta)$, depends only on $\delta$ and not on $E$ or $n$.
\end{proof}

Taking $\delta=d^\star(E)$ gives the density-dependent formulation
asked for in Fish's original problem.

\section{Applications}
We first deduce the square-difference corollaries from the integer theorem, then treat finite cyclic rings, and finally prove the transfer needed for the rational consequences.

\begin{proof}[Proof of Corollary~\ref{cor2}]
Let $T(x,y)=(x-y,x+y)$ and $E'=T(E)$. Since $\det T=2$,
Lemma~\ref{lem:density-facts} gives
\[
 d^\star(E')=\frac12d^\star(E)\geq\frac{\delta}{2}.
\]
Apply Theorem~\ref{thm:main} to $E'$ and set
\[
 K'(\delta)=K(\delta/2).
\]
Linearity of $T$ gives $E'-E'=T(E-E)$, and hence
\[
 P(E')
 =\{(x-y)(x+y):(x,y)\in E-E\}
 =\{x^2-y^2:(x,y)\in E-E\}.
\]
The ideal containment for $P(E')$ proves the claim.
\end{proof}

\begin{proof}[Proof of Corollary~\ref{cor1}]
Let
\[
 E=E_1\times E_2.
\]
By \eqref{eq:density-product},
\[
 d^\star(E)=\alpha\beta,
\]
and
\[
 E-E=(E_1-E_1)\times(E_2-E_2).
\]
Applying Corollary~\ref{cor2} to $E$ with density threshold
$\alpha\beta$ gives
\[
 K'(\alpha\beta)\mathbb Z
 \subseteq
 \{x^2-y^2:
     x\in E_1-E_1,\ y\in E_2-E_2\}.
\]
Since
\[
 K'(\alpha\beta)=K(\alpha\beta/2),
\]
the conclusion follows by taking
\[
 K(\alpha,\beta)=K(\alpha\beta/2).
\]
\end{proof}

\begin{proof}[Proof of Corollary~\ref{corfirst:abelian}]
Let $\pi_N:\mathbb Z^2\to\mathbb Z_N^2$ be reduction modulo $N$ and set $\widetilde E=\pi_N^{-1}(E)$. This is a periodic set with
\[
 d^\star(\widetilde E)=\frac{|E|}{N^2}\geq\delta.
\]
Theorem~\ref{thm:main} gives $K(\delta)\mathbb Z\subseteq P(\widetilde E)$. As reduction modulo $N$ is a ring homomorphism and $\pi_N(\widetilde E-\widetilde E)=E-E$, reduction of $P(\widetilde E)$ is exactly
\[
 \{uv\pmod N:(u,v)\in E-E\}.
\]
The image of $K(\delta)\mathbb Z$ in $\mathbb Z_N$ is $d\mathbb Z_N$, where $d=\gcd(K(\delta),N)$. Thus $d\mid N$, $d\leq K(\delta)$, and the required containment follows.
\end{proof}

\begin{proof}[Proofs of Corollaries~\ref{corsecond:abelian} and~\ref{corthird:abelian}]
For Corollary~\ref{corsecond:abelian}, let $\widetilde E_i\subseteq\mathbb Z$ be the inverse image of $E_i$ under reduction modulo $N$. Their densities are $|E_i|/N$, so the product $\widetilde E_1\times\widetilde E_2$ has density at least $\alpha\beta$. Applying Corollary~\ref{cor2} to this product yields the integer square-difference containment with modulus $k=K'(\alpha\beta)=K(\alpha\beta/2)$. Reducing modulo $N$ gives the claim with $d=\gcd(k,N)\leq k$.

For Corollary~\ref{corthird:abelian}, lift $E\subseteq\mathbb Z_N^2$ to $\widetilde E=\pi_N^{-1}(E)$ and apply Corollary~\ref{cor2} with threshold $\delta$. The polynomial $x^2-y^2$ commutes with reduction modulo $N$, so projection gives the required ideal with $d=\gcd(K'(\delta),N)\leq K'(\delta)$.
\end{proof}

For the rational consequences, we first prove the finite-window transfer that is needed before applying rational stretching.
\begin{lemma}\label{lem:rational-transfer}
Let $0<\delta\leq1$, and let $K_{\mathbb Z}(\cdot)$ denote a choice of moduli in Theorem~\ref{thm:main}. Put $K=K_{\mathbb Z}(\delta/4)$. Every set $E\subseteq\mathbb Q^2$ with $d^\star(E)\geq\delta$ satisfies $K\in P(E)$.
\end{lemma}

\begin{proof}
By the finite-window assertion of Lemma~\ref{lem:density-facts}, for each $L\geq1$ there is $z_L\in\mathbb Q^2$ such that
\[
 B_L=(E-z_L)\cap Q_L\subseteq\mathbb Z^2,
 \qquad |B_L|\geq\delta L^2.
\]
Suppose that $K\notin P(E)$ and let
\[
 H_K=\{(a,b)\in\mathbb Z^2:ab=K\}.
\]
Then $(B_L-B_L)\cap H_K=\varnothing$, since $B_L-B_L\subseteq E-E$. As $K\geq1$, every $v\in H_K$ satisfies $\|v\|_\infty\leq K$.

Set $q_L=L+2K+1$ and $\widetilde B_L=B_L+q_L\mathbb Z^2$. Differences within one period block lie in $B_L-B_L$. A difference between distinct blocks has a coordinate of absolute value at least
\[
 q_L-(L-1)=2K+2>K.
\]
Thus $(\widetilde B_L-\widetilde B_L)\cap H_K=\varnothing$. On the other hand,
\[
 d^\star(\widetilde B_L)=\frac{|B_L|}{q_L^2}
 \geq\delta\left(\frac{L}{L+2K+1}\right)^2
 \geq\delta/4
\]
for all sufficiently large $L$. Theorem~\ref{thm:main}, applied to the integer set $\widetilde B_L$ with threshold $\delta/4$, gives $K\in P(\widetilde B_L)$, a contradiction.
\end{proof}

\begin{proof}[Proof of Corollary~\ref{rational1}]
Let $\delta=d^\star(E)>0$ and $K=K_{\mathbb Z}(\delta/4)$ as in Lemma~\ref{lem:rational-transfer}. Fix $q\in\mathbb Q\setminus\{0\}$ and set
\[
 E'=L_{K/q}(E),\qquad L_{K/q}(x,y)=\left(\frac Kq x,y\right).
\]
This is an automorphism of the discrete additive group $\mathbb Q^2$, so Lemma~\ref{lem:density-facts} gives $d^\star(E')=\delta$. Lemma~\ref{lem:rational-transfer}, now applied to $E'$, yields $K\in P(E')$. We do not apply the integer theorem directly to $E'$.

Since $E'-E'=L_{K/q}(E-E)$, we have
\[
 P(E')=\frac Kq P(E).
\]
Thus $K\in(K/q)P(E)$ implies $q\in P(E)$. This works for every nonzero rational $q$, and $0\in P(E)$ because $E\neq\varnothing$. Therefore $P(E)=\mathbb Q$.
\end{proof}

\begin{proof}[Proof of Corollary~\ref{rational2}]
By Lemma~\ref{lem:density-facts}, $E=E_1\times E_2$ has positive upper Banach density in $\mathbb Q^2$. The map $T(x,y)=(x-y,x+y)$ is an additive group automorphism, with inverse
\[
 T^{-1}(u,v)=\left(\frac{u+v}{2},\frac{v-u}{2}\right).
\]
Thus $E'=T(E)$ has the same positive density. Corollary~\ref{rational1} gives $P(E')=\mathbb Q$, and linearity gives
\[
 P(E')=\{(x-y)(x+y):x\in E_1-E_1,\ y\in E_2-E_2\}
       =(E_1-E_1)^2-(E_2-E_2)^2.\qedhere
\]
\end{proof}

\begin{proof}[Proof of Corollary~\ref{rational3}]
The automorphism $T(x,y)=(x-y,x+y)$ preserves upper Banach density by Lemma~\ref{lem:density-facts}. Apply Corollary~\ref{rational1} to $E'=T(E)$. Since $E'-E'=T(E-E)$,
\[
 \mathbb Q=P(E')=\{x^2-y^2:(x,y)\in E-E\}.\qedhere
\]
\end{proof}

\vspace{.5in}

\noindent
\textbf{Sayan Goswami}\\
Ramakrishna Mission Vivekananda Educational and Research Institute, Belur Math,\\
Howrah, West Bengal--711202, India.\\
Email: \href{mailto:sayan92m@gmail.com}{sayan20math@gmail.com}\\

\noindent
\textbf{Chunlin Liu}\\
School of Mathematical Sciences, Dalian University of Technology,
Dalian 116024, P.~R. China;\\
and Institute of Mathematics, Polish Academy of Sciences,
ul. \'{S}niadeckich 8, 00-656 Warsaw, Poland\\
Email: \href{mailto:chunlinliu@mail.ustc.edu.cn}{chunlinliu@mail.ustc.edu.cn};
\href{mailto:chunlin.liu.math@gmail.com}{chunlin.liu.math@gmail.com}

\end{document}